\documentclass[final,12pt,english]{elsarticle}

\usepackage{amssymb}
\usepackage{amsthm}

\biboptions{numbers}
\usepackage{pifont}
\usepackage{xspace}
\def\G1{\hbox{$\displaystyle{\mbox{\ding{172}}}$}}

\def\fileversion{2.1}
\def\filedate{1998/11/15}
\wlog{LaTeX `preamble'  \fileversion\space<\filedate> (RDL)}
\makeatletter

\usepackage{textcomp}
\usepackage{pifont}
\usepackage{amsmath}

\newcommand{\oneto}[1]{1, \ldots, {#1}}

\newcommand{\ie}{i.e., }

\newcommand{\insmat}[1]{\mathop{\mbox{\rm #1}}}  
\def\1{{\mathchoice {\roman {1\mskip-4mu l}} {\roman {1\mskip-4mu l}}
{\roman {1\mskip-4.5mu l}} {\roman {1\mskip-5mu l}}}}

\newcommand{\ip}[2]{{#1}^T{#2}}

\newcommand{\norm}[1]{\left\| #1 \right\|}

\newcommand{\goes}{\rightarrow}

\@ifundefined{grad}
{
\newcommand{\grad}{\nabla}
}{}

\newcommand{\union}{\cup}

\newcommand{\ldef}{:=}
\newcommand{\rdef}{=:}

\newcommand\refe[1]{(\ref{#1})}

\@ifundefined {flalign}
{

   \def\R{\inmat{\rm I}\!\inmat{\rm R}}

   \def\N{\inmat{\mathrm I}\!\inmat{\mathrm N}}
}
{
    \newcommand{\field}[1]{\mathbb #1}
    
    \def\R{\field{R}}

    \def\N{\field{N}}
}

\def\openA{{{\mathrm A}\kern-.63em{\mathrm A}}}
\def\openB{{{\mathrm I}\kern-.16em {\mathrm B}}}
\def\openC{{\mathrm C\kern-.18cm\vrule width.6pt height 6pt depth-.2pt \kern.18cm}}
\def\openD{{{\mathrm I}\kern-.16em {\mathrm D}}}
\def\openF{{{\mathrm I}\kern-.16em {\mathrm F}}}
\def\openI{{{\mathrm I}\kern-.2em {\mathrm I}}}
\def\openH{{{\mathrm I}\kern-.16em {\mathrm H}}}
\def\openK{{{\mathrm I}\kern-.16em {\mathrm K}}}
\def\openN{{{\mathrm I}\kern-.16em {\mathrm N}}}
\def\openP{{{\mathrm I}\kern-.16em {\mathrm P}}}
\def\openQ{{\mathrm Q\kern-.21cm\vrule width.6pt height 6.2ptdepth-.2pt \kern.21cm}}
\def\openQ{{{\mathrm Q}\kern-.63em {\mathrm Q}}}
\def\openR{{{\mathrm I}\kern-.16em {\mathrm R}}}
\def\openS{{{\mathrm S}\kern-.68em{\mathrm S}}}
\def\openT{{{\mathrm T}\kern-.30em {\mathrm T}}}
\def\openZ{{{\mathrm Z}\kern-.28em{\mathrm Z}}}

\@ifundefined{flalign}
{
}
{

}

\@ifundefined{flalign}
{

}
{
  \RequirePackage{amsgen}

}

\@ifundefined{diag}
{

}{}

\@ifundefined{flalign}
{
   \def\Rank{\mathop{\insmat{rank}}}

}
{
  \RequirePackage{amsgen}
  \DeclareMathOperator*{\Rank}{rank}

}
\@ifundefined{rank}
{
\newcommand{\rank}[1]{\Rank #1}
}{}

\newcommand{\ba}{\begin{array}}
\newcommand{\ea}{\end{array}}
\newcommand{\bc}{\begin{center}}
\newcommand{\ec}{\end{center}} \newcommand{\ds}{\displaystyle}

\newcommand{\be}{\begin{equation}}
\newcommand{\ee}{\end{equation}}

\newcommand{\problem}[4]
    {\ba{lc} {\ds #4_{#1}} & {#2} \\*
        {\subjectto} & {\ba[t]{c} #3 \ea}
     \ea}
\newcommand{\eprobmin}[4]
    {\begin{equation} \problem{#1}{#2}{#3}{\minimize} \label{#4}
         \end{equation}}

\@ifundefined {flalign}
{

} { }

\makeatletter
\@ifundefined {flalign}
{

}{ }
\makeatother

\newcommand{\xbar}{\bar{x}}
\newcommand{\pibar}{\bar{\pi}}
\newcommand{\mubar}{\bar{\mu}}

\def\G1{\hbox{$\displaystyle{\mbox{\ding{172}}}$}\xspace}
\newcommand {\grossone}{\G1}

\journal{Applied Mathematics and Computation}
\newtheorem{definition}{Definition}
\newtheorem{assumption}{Assumption}

\begin{document}

\begin{frontmatter}



\title{The use of Grossone in Mathematical Programming and Operations Research}


\author[unicam]{Sonia De Cosmis}
\ead{sonia.decosmis@unicam.it}
\author[unicam]{Renato De Leone\corref{cor1}}
\ead{renato.deleone@unicam.it}
\address[unicam]{School of Science and Technology, Universit\`a di Camerino, via Madonna delle Carceri 9, Camerino (MC) 62032 Italy}
\cortext[cor1]{Corresponding author}

\begin{abstract}
The concepts of infinity and infinitesimal in mathematics date back to anciens Greek and have always attracted great attention. Very recently, a new methodology has been proposed by Sergeyev \citep{Sergeyev:2} for performing calculations with infinite and infinitesimal quantities, by introducing an  infinite unit of measure expressed by the numeral \G1 (grossone). An important characteristic of this novel approach is its attention to numerical aspects. In this paper we will present some possible applications and use of \G1 in Operations Research and Mathematical Programming. In particular, we will show how the use of \G1 can be beneficial in anti--cycling procedure for the well--known simplex method for solving Linear Programming Problems and in defining exact differentiable Penalty Functions in Nonlinear Programming.
\end{abstract}


\end{frontmatter}


\section{Introduction}
A novel approach to infinite and infinitesimal numbers has been recently proposed by Sergeyev in a book and in a series of papers \citep{Sergeyev:2,Sergeyev:1,Sergeyev:3,Sergeyev:4}.
By introducing a new infinite unit of measure (the numeral  \textit{grossone}, indicated by  $\G1$) as the number of elements of the set  of the natural numbers, he shows that it is possible to effectively work with infinite and infinitesimal quantities and to solve many problems connected  to them  in the field of  applied and  theoretical mathematics.
In this new system, there is the opportunity to treat infinite and infinitesimal numbers as particular cases of a single structure, offering a new view and alternative approaches to important aspects of mathematics such as sums of series (in particular, divergent series), limits, derivatives, etc.

The new numeral \textit{grossone} can be  introduced by describing its properties (in a similar way as done in the past  with the introduction of 0 to switch from natural to integer numbers). The Infinity Unit Axiom  postulate (IUA) \citep{Sergeyev:1,Sergeyev:2} is  composed of three parts: Infinity, Identity, and Divisibility:
\begin{itemize}
  \item \textit {Infinity.} Any finite natural number $n$ is less than grossone, \ie $n < \grossone $.
  \item \textit {Identity.} The following relationships link \grossone\; to the identity elements $0$ end $1$
   \begin{equation}
    0 \cdot \grossone = \grossone \cdot 0 = 0 ,\; \grossone - \grossone =0,\; \frac{\grossone}{\grossone}=1 , \; \grossone ^{0} =1 ,\; 1^{\grossone} =1 , \quad 0^{\grossone} = 0
  \end{equation}
  \item  \textit {Divisibility.} For any  finite natural number $n$, the sets $\N _{k,n},1 \leq k \leq n $,
  \begin{equation}
      \N_{k,n}={k,k+n,k+2n,k+3n,....}, \quad 1\leq k\leq n, \quad \bigcup_{k=1}^{n} \N_{k,n}= \N
  \end{equation}
   have the same number of elements indicated by  $\frac{\grossone}{n} $.

\end{itemize}
The axiom above states that the  infinite number \grossone, greater than any finite number, behaves as any natural number with the elements $0$ and $1$. Moreover, the quantities $\frac{\grossone}{n} $ are integers for any natural $n$.
This axiom is added to the standard  axioms of  real numbers and, therefore, all standard properties (commutative, associative, existence of inverse, etc.) also apply to $\G1$.

Sergeyev \citep{Sergeyev:3,Sergeyev:4} also defines a new way to express the infinite and infinitesimal numbers using a register similar to traditional positional number system, but with base number \grossone. A number $\mathbf{C}$ in this new system can be constructed  by subdividing it  into groups corresponding to powers of \grossone\; and has the following representation:
\begin{equation}
\mathbf{C} = c_{p_{m}}\grossone^{p_{m}} +....+c_{p_{1}}\grossone^{p_{1}}+c_{p_{0}}\grossone^{p_{0}}+c_{p_{-1}}\grossone^{p_{-1}}+....+c_{p_{-k}}\grossone^{p_{-k}}.
\label{somma}
\end{equation}
where the quantities $c_{i}$ (the \textit{grossdigits}) and $p_{i}$ (the \textit{grosspowers}) are expressed by the traditional numerical system for representing finite numbers (for example, floating point numbers).  The grosspowers are sorted in descending order:
$$p_{m} > p_{m-1} > .... > p_{1} > p_{0} > p_{-1} >... p_{-(k-1)} > p_{-k}$$
with  $p_{0} = 0$.

In this new numeral system, finite numbers are represented by numerals with only one grosspower $p_{0} = 0 $.
Infinitesimal numbers are represented by numeral $ \mathbf{C}$  having only negative finite
or infinite grosspowers. The simplest infinitesimal number is $\grossone^{-1}$ for which
\begin{equation}
  \G1^{-1} \grossone=\grossone\;\grossone^{-1}=1.
 \label{inversa moltiplicazione}
\end{equation}
We note that infinitesimal numbers are not equal to zero. In particular, $ \frac{1}{\grossone}>0$.
Infinite numbers are expressed by numerals having at least one finite or infinite grosspower greater than zero.

A peculiar characteristic of the newly proposed numeral system is its attention to its numerical aspects and to applications.
The Infinity Computer proposed by Sergeyev is able to execute computations with infinite, finite, and infinitesimal numbers numerically (not symbolically) in a novel framework.

In this paper we will present two possible uses of this numeral system in Mathematical Programming and Operations Research. In particular, in Section \ref{sect:simplex} we will show a simple way to implement anti--cycling strategies in the simplex method for solving Linear Programming problems. Various anti--ciclyng procedures have been proposed and implemented in state--of-the--art softwares. The lexicographic strategies has received particular attention since it allows, in contrast to Bland's rule, complete freedom in choosing, at each iteration,  the entering variable.
In Section \ref{sect:nonlinear} we revert our attention to Nonlinear Programming problems and, in particular, to differentiable penalty functions. In the new numeral system it is possible to define an exact, differentiable penalty function and we will show that stationary points of this penalty function are KKT points for the original Nonlinear Programming problem. Two simple examples are also provided showing the effectiveness of this approach. Conclusions and indications for further applications of \G1 in Mathematical Programming are reported in Section \ref{sect_conclusions}.

We briefly describe our notation now. All vectors are column vectors
and will be indicated with lower case Latin letter ($x$, $z$,
$\ldots$). Subscripts indicate components of a vector, while
superscripts are used to identify different vectors.
Matrices will be indicated with upper case roman letter ($A$,
$B$, $\ldots$). If $A \in \R^{m \times n}$, $A_{.j}$ is the $j$--th column of $A$; if $B \subseteq \left\{\oneto{n}\right \}$, $A_{.B}$ is the submatrix of $A$ composed by all columns $A_{.j}$ such that $j \in B$.
The set of real numbers and the set of nonnegative real numbers will be denoted by
$\R$ and $\R_+$  respectively. The rank of a matrix $A$ will be indicated by $\rank{A}$.
The space of the $n$--dimensional vectors with real components
will be indicated by $\R^n$ and $\R^n_+$ is an abbreviation for the nonnegative orthant in $\R^n$.
The symbol $\norm{x}$ indicates the Euclidean norm of a vector $x$.
Superscript $^T$ indicates transpose.
The scalar product of two vectors $x$ and $y$ in $\R^n$ will be denoted by $\ip{x}{y}$. Here and throughout the symbols $\ldef$ and $\rdef$ denote definition of the term on the left and the right sides of each symbol,  respectively.
The gradient $\nabla f(x)$ of a continuously differentiable function
 $f:\R^n \goes \R$ at a point $x\in \R^n$ is assumed to be a column vector.
 If $F:\R^n \goes \R^m$ is a continuously differentiable
 vector--valued function, then $\nabla F(x)$ denotes the Jacobian matrix
 of $F$ at $x\in \R^n$.

\section{Lexicograhic rule and grossone}
\label{sect:simplex}
The simplex method, originally proposed by G.B. Dantzig \citep{Dantzig} more than half a century ago, is still today one of
the most used algorithms for solving Linear Programming problems.
Finite termination of the method can only be guaranteed if special techniques are employed to eliminate cycling. In this section we will show how in the new numeral system it is very simple to implement such anti--cycling rules.

Given a matrix $A \in \R^{m\times n}$ with $\rank{A} = m$ and vectors $b \in \R^m$ and $c \in \R^n$, the Linear Programming problem in standard form can be stated as follows
\eprobmin{x}{\ip{c}{x}}{Ax = b \\ x \geq 0.}{lp:pri}

The simplex algorithm  moves from a Basic Feasible Solution (BFS) to an adjacent Basic Feasible Solution  until an optimal solution is reached or unboundness of the problem is detected.

More precisely, a submatrix $A_{B.}\in \R^{m\times m}$ is a  basis matrix if it is nonsingular; a point $\xbar \in \R^n$ is a BFS if it is feasible and the columns of $A$ corresponding to positive components of $\xbar$ are linearly independent. Basic Feasible Solutions correspond to vertices of the feasible region. A vertex $\xbar$ is non--degenerate if exactly $m$ components of $\xbar$ are positive: in this case there is a single Basis matrix associated to the point. A vertex $\xbar$ is degenerate if fewer than $m$ components are strictly positive: in this case more than one basis matrix can be associated to the point.

Finding an initial Basic Feasible Solution, if it exists, requires to solve (always using the Simplex Method) an auxiliary problem, whose initial solution is trivially obtained adding artificial variables.

A single iteration of the (primal) simplex method requires the following steps:

\begin{enumerate}
\item[]{\bf Step 0} Let $B \subseteq \left\{\oneto{n} \right \}$ be the  {\it
current} base and let $x \in X$ the {\it current } BFS

\[ x_B = A_{.B}^{-1}b \geq 0, \quad x_N = 0, \quad |B| = m. \]

Assume

\[ B = \left \{j_1, j_2, \ldots, j_m \right \} \]
 and
\[ N = \left\{\oneto{n} \right \} \setminus B = \left \{j_{m+1}, \ldots, j_n \right
\}. \]

\item[]{\bf Step 1}  Compute
\[ \pi = A_{.B}^{-T}c_B \]
and the reduced cost vector
\[ \bar{c}_{j_k} = c_{j_k} - \ip{A_{.j_k}}{\pi}, \quad k = m+1, \ldots, n. \]

\item[]{\bf Step 2} If
\[ \bar{c}_{j_k} \geq 0, \quad \forall k = m+1, \ldots, n \]
the current      point is an optimal BFS and the algorithm stops.

Instead, if  $\bar{c}_N \not\geq 0$, choose  $j_r$ with  $r \in \left \{m+1, \ldots, n \right \}$ such that $\bar{c}_{j_r} < 0$. This is the variable {  candidate to enter the base}.

\item[]{\bf Step 3} Compute
\[ \bar{A}_{.j_r} = A_{.B}^{-1} A_{.j_r} \]

\item[]{\bf Step 4} If
\[ \bar{A}_{.j_r} \leq 0 \]
the problem is unbounded below and the algorithm stops.

Otherwise, compute

 $$\bar{\rho} = \min_{i : \bar{A}_{ij_{r}}>0} \left \{ \frac{\left(A_{.B}^{-1}b \right)_i}{\bar{A}_{ij_{r}}} \right \}$$
and let $s\in \left \{\oneto{m}\right \}$ such that
$ \ds \left \{ \frac{\left(A_{.B}^{-1}b \right )_s}{\bar{A}_{sj_{r}}} \right \}= \bar{\rho} $; $j_s$ is the {  leaving variable}.

\item[]{\bf Step 5} Define
\[ \ba{lcll}
 \bar{x}_{j_k} & = & 0, & k = m+1, \ldots, n, k \neq r  \\
 \bar{x}_{j_r}& = & \bar{\rho} \\
 \bar{x}_B(\rho)    & = & A_{.B}^{-1}b - \bar{\rho} \bar{A}_{.j_r}.\ea \]
and
\[ \bar{B} = B \setminus \{j_s\} \union \{j_r \} =
  \left \{j_1, j_2, \ldots, j_{s-1}, j_r, j_{s+1}, \ldots, j_m \right \} \]

\end{enumerate}

Note that, when a non--degenerate step is performed, the objective function value strictly decreases. Therefore, the Simplex Method will terminate after a finite number of steps if all the BFS are non--degenerate.
In case of degenerate BFS, the objective function value may remain the same for a number of steps and the algorithm will cycle. Therefore, specific anti--cycling rules must be implemented to avoid this negative feature.

Among the various anti--cycling criteria, the lexicographic pivoting rule \citep{gow:55,Terlaky} has received special attention since, in contrast to other rules such as Bland's rule, there is a complete freedom in choosing the entering variable.

The lexicographic simplex method requires, at each iteration, to choose the leaving variable using a specific procedure (the lexicographic rule) \citep{Terlaky}.

Let $B_0$ be the initial base and $N_0 = \left\{\oneto{n} \right \} \setminus B_0$.
We can always assume, after columns reordering, that $A$ has the form
\[
A = \left[ A._{B_0} \quad \vdots \quad A._{N_0} \right]
\]
Let $$ \bar{\rho}= \min_{i : \bar{A}_{ij_{r}}>0}\frac{(A._{B}^{-1}b)_{i}}{\bar{A}_{ij_{r}}}$$
if such minimum value is reached in only one index, this is the leaving variable.
Otherwise, let
\[ {\cal I}_1 \ldef \left \{i \in \left\{\oneto{m}\right\}: \bar{A}_{ij_{r}}>0 \insmat{ and }  \frac{(A._{B}^{-1}b)_{i}}{\bar{A}_{ij_{r}}}= \bar{\rho} \right \} \]
and
\[ \bar{\rho}_1 \ldef \min_{i \in {\cal I}_1}\frac{(A._{B}^{-1}A._{B_0})_{i1}}{\bar{A}_{ij_{r}}}. \]
and choose the index $i_1 \in {\cal I}_1$ getting the minimum, that is the index $i_1$ such that
\[ \frac{(A._{B}^{-1}A._{B_0})_{{i_1}1}}{\bar{A}_{{i_1}j_{r}}} = \bar{\rho}_1. \]
If the  minimum is reached by only one index $i_1$, then this is the leaving variable. Otherwise, let
\[ {\cal I}_2 \ldef \left \{i \in {\cal I}_1: \frac{(A._{B}^{-1}A._{B_0})_{{i}1}}{\bar{A}_{{i}j_{r}}} = \bar{\rho}_1 \right \} \insmat{   and  }
 \bar{\rho}_2 \ldef \min_{i \in {\cal I}_2}\frac{(A._{B}^{-1}A._{B_0})_{i2}}{\bar{A}_{ij_{r}}} \]
and choose the index $i_2 \in {\cal I}_2$ getting the minimum, that is the index $i_2$ such that
\[ \frac{(A._{B}^{-1}A._{B_0})_{{i_2}2}}{\bar{A}_{{i_2}j_{r}}} = \bar{\rho}_2. \]

This procedure will terminate providing a single index since  the rows of the matrix $(A_{.B}^{-1}A_{.{B_0}})$ are linearly independent.
The finiteness of the lexicographic simplex method follows from the simple observation that the vector whose first element is the current objective function value and the other components are the reduced costs, strictly lexicographically decreases  at each iteration.

The procedure outlined above is equivalent to
perturb each component of the RHS vector $b$ by a very small quantity \citep{Charnes}.

If this perturbation is small enough, the new Linear Programming problem is nondegerate and the simplex method produces exactly the same pivot sequence as  the lexicographic pivot rule.

However, is very difficult to determine how small this perturbation must be. More often a symbolic perturbation is used (with higher computational costs).

In the new numeral system obtained by the introduction of \G1 we propose to replace $b_i$ by
\begin{equation} \widetilde{b}_{i} =  b_{i} + \sum_{j \in B_{0}} A_{ij} \G1^{-j}. \label{b:mod} \end{equation}
More specifically, let
\[ e =  \left[ \begin{array}{c}
 \G1^{-1} \\
 \G1^{-2} \\
 \vdots \\
 \G1^{-m}  \\
  \end{array}
 \right] \]
and define
\begin{equation}\overline{\widetilde{b}} = A._{B}^{-1} ( b + A._{B_0}e ) = A._{B}^{-1}b + A._{B}^{-1}A._{B_0} e. \label{bb:mod} \end{equation}
Therefore, $\overline{\widetilde{b}}_i = (A._{B}^{-1}b)_{i} + \ds \sum_{k=1}^{m}( A._{B}^{-1}A._{B_0})_{i_{k}} \G1^{-k}$
 and
\begin{eqnarray}
 \lefteqn{\min_{i: \bar{A}_{ij_{r}}>0}\frac{(A._{B}^{-1}b)_{i}+\ds \sum_{k=1}^{m}( A._{B}^{-1}A._{B_0})_{i_{k}} \G1^{-k}}{\bar{A}_{ij_{r}}} }  \nonumber\\
 & = & \min_{i : \bar{A}_{ij_{r}}>0} \frac{(A._{B}^{-1}b)_{i}}{\bar{A}_{ij_{r}}}+ \frac {( A._{B}^{-1}A._{B_0})_{i{1}}} {\bar{A}_{ij_{r}}}\G1^{-1}+ \ldots + \frac{( A._{B}^{-1}A._{B_0})_{i{m}}}
  {\bar{A}_{ij_{r}}} \G1^{-m}
 \label{grossone:choose} \end{eqnarray}
Taking into account the properties of the power of \G1, the index $i$ that will be chosen by the formula
\refe{grossone:choose} will be exactly the same obtained by the lexicographic pivoting rule outlined before.

\section{Nonlinear programming} \label{sect:nonlinear}

Nonlinear constrained optimization problems  are an important class of problems with a broad range of engineering, scientific, and operational applications. The problem can be stated as follows:

\eprobmin{x}{f(x)}{g(x) \leq 0 \\ h(x) = 0}{nlp:gen}
where  $f: \R^n \goes \R$,   $g:\R^n \goes \R^m$ , and $h:\R^n \goes \R^p$. For simplicity we will assume that the functions $f$, $g$, and $h$  are twice continuously differentiable. The feasible set will be indicated by $X$.
In this section we will show how, using the new numeral system, it is possible to introduce exact differentiable penalty functions and we will discuss the relationship between stationary point of the penalty function and KKT points of the original constrained problem. After introducing some basic concepts and definitions, we consider in Subsection \ref{subsect:eq}   the equality constrained case and later, in Subsection \ref{subsect:ineq}, the most general case of equality and inequality constraints.

\begin{definition}
Given a point $x^0 \in X$, the  set of active constraints at $x^0$ is
 \[ I(x^0) = \left \{1, \ldots p \right \} \cup \left\{i \in \left\{ 1,\ldots m \right\}: g_{i}(x^0) = 0\right\}.\]
\end{definition}

\begin{definition}
The linear independence constraint qualification
(LICQ) condition  is said to hold true at $x^0\in X$ if the set of gradients of the active constraints at $x^0$ is linearly independent.
\end{definition}

The associated Lagrangian function $L(x, \mu, \pi)$ is :
\begin{equation}
L(x, \pi, \mu)  \ldef  f(x) + \ip{\mu}{g(x)} + \ip{\pi}{h(x)}
\end{equation}
where $\pi \in \R^p$ and $\mu \in \R^m$ are the multipliers associated to the equality and inequality constraints respectively.
\begin{definition}
A triplet $\left(x^*, \pi^* , \mu^* \right)$ is a Karush-–Kuhn–-Tucker (KKT) point if
\begin{eqnarray*}
\grad_x L(x^*, \mu^*, \pi^*) = \grad f(x^*) + \sum_{i=1}^m  \grad g_i(x^*) \mu_i^* + \sum_{j=1}^p  \grad h_j(x^*) \pi_j^* &= &0 \\
\grad_\mu L(x^*, \mu^*, \pi^*) = g(x^*) & \leq & 0  \\
\grad_\pi L(x^*, \mu^*, \pi^*)  =  h(x^*) &= &0 \\
\mu^* &  \geq & 0 \\
\ip{ \mu^*}{\grad_\mu L(x^*, \mu^*, \pi^*)} &  =& 0
\end{eqnarray*}
\end{definition}
The following theorem \citep{mangasarian.69:nonlinear} states  the first order necessary optimality conditions for the nonlinear optimization problems.

\begin{thm}
Consider the nonlinear optimization problem \refe{nlp:gen}. Let $x^*\in X$ be a local minimum of Problem \refe{nlp:gen}, and assume that at  $x^*$ the LICQ\footnote{Weaker Constraint Qualification conditions can be imposed. See \citep{solodov.cq:10} for a review of different Constraints Qualification conditions and the relationship among them.} condition holds true. Then, there exist vectors $\mu^*$ and $\pi^*$ such that the triplet $\left(x^*, \pi^* , \mu^* \right)$ is a Karush-–Kuhn–-Tucker point.
\end{thm}

Different algorithms have been proposed and studied for solving the general nonlinear optimization problems. We refer the interested reader to the classical book of Fletcher \citep{fletcher.90:practical} and to the more recent book of Boyd and Vandenberghe \citep{Boyd.Vandenberghe.04:convex} for a general discussion of Nonlinear Programming  algorithms. Here we will concentrate our attention on exact penalty functions and we will show how the use of $\G1$ can be beneficial in defining exact differentiable penalty functions.

Let $\phi: x \in \R^n \goes \R_+$ be a continuously differentiable function such that

\[ \phi(x)  \left \{ \ba{ll} = 0 & \insmat{   if  } x \in X  \\
                          > 0 & \insmat{   otherwise} \ea \right . \]

\noindent A possible choice for the function  $\phi(x)$ is
\[ \phi(x) =  \ds \sum_{i=1}^m \insmat{max} \bigl\{g_{i}(x), 0 \bigr \}^2 + \sum_{j=1}^k h_{j}(x)^2 \]
Note that this function is continuously differentiable but not twice differentiable.

From the function $\phi(x)$, the  exterior penalty function
\[ P( x,  \epsilon) = f(x) + \frac{1}{ 2 \epsilon}\phi(x) \]
can be constructed and the following unconstrained optimization problem can be defined:
\begin{equation} \minimize_{x}{P( x,  \epsilon).}\label{nlp:penalty} \end{equation}

It can be shown that there is no finite value of the penalty parameter  $\epsilon$ for which, by solving Problem \refe{nlp:penalty},  a solution of the original Problem \refe{nlp:gen} is obtained. Sequential penalty methods require to solve a sequence of minimization problems (similar to Problem \refe{nlp:penalty}) for decreasing values of the parameter $\epsilon$. It is possible to construct exact non--differentiable penalty functions \citep{fletcher.90:practical} as well as differentiable exact penalty function but, in this case,  it is necessary to introduce terms related to first order optimality conditions \citep{dipillo.grippo:exact}, thus making the penalty function much more complicate.

%
%
%

We propose here to substitute the term $1/\epsilon$ with $\G1$. In the next subsection we will present convergence  results for the simpler case of equality constraints, while Subsection \ref{subsect:ineq} will present results and a simple example of exact penalty function for the more general case.

\subsection{The equality constraint case.} \label{subsect:eq}
Consider  the  optimization problem with equality constraints:
\eprobmin{x}{f(x)}{ h(x) = 0}{nlp:eq}
where $f$ and $h$ are defined as before.

The following theorem \citep[Theorem 12.1.1]{fletcher.90:practical} states the convergence for the sequential penalty method.

\begin{thm}
Let $f(x)$ be bounded below in the (nonempty) feasible region, and let $\{\epsilon_k\}$ be a monotonic non-increasing sequence such that $\{\epsilon_k\} \downarrow 0$, and assume that for each $k$ there exists a global minimum $x(\epsilon_k)$ of $P( x; \epsilon_k)$. Then

\begin{itemize}
  \item[(i)]  $ \{ P( x(\epsilon_k); \epsilon_k) \}$ is monotonically non--decreasing
  \item[(ii)] $ \{  \phi(x( \epsilon_k )) \}$ is monotonically non--increasing
 \item[(iii)] $f(x(\epsilon_k))$ is monotonically non--decreasing
\end{itemize}
Moreover, $\lim_k h(x(\epsilon_k)) = 0$ and each limit point $x^*$ of the sequence $\{x(\epsilon_k)\}_k$ solves Problem \refe{nlp:eq}.
\end{thm}

In order to derive first order optimality condition, we will make the following assumptions on the functions $f(x)$ and $h(x)$, on the gradient $\grad f(x)$ and on the Jacobian $\grad h(x)$:

\begin{assumption}\label{assumpt:expand}
If
\[  x = x^{0} + \G1^{-1} x^{1} + \G1^{-2}x^{2}+\ldots \]
with $x^i  \in \R^n$, then
\begin{eqnarray*}
 f(x) &= & f(x^0) + \G1^{-1} f^{(1)} (x) + \G1^{-2}f^{(2)}(x) +\ldots \\
 h(x) &= & h(x^0) + \G1^{-1} h^{(1)} (x) + \G1^{-2}h^{(2)}(x) +\ldots \\
\nabla f(x) & = & \nabla f(x^{0})+ \G1^{-1} F^{(1)} (x) + \G1^{-2}F^{(2)}(x) + \ldots \\
\nabla h(x) & = & \nabla h(x^{0})+ \G1^{-1} H^{(1)} (x) + \G1^{-2}H^{(2)}(x) + \ldots
\end{eqnarray*}
where $f^{(i)}: \R^n \goes \R$,   $h^{(i)}:\R^n \goes \R^p$, $F^{(i)}: \R^n \goes \R^n$, $H^{(i)}:\R^n \goes \R^{p\times n}$ are all finite--value functions.
\end{assumption}
In the sequel, we will always suppose that the assumption above holds. Note that  Assumption 1 is satisfied (for example) by functions that are product of polynomial functions in a single variable, \ie $$p(x) = p_1(x_1) p_2(x_2) \cdots p_n(x_n)$$ where $p_i(x_i)$ is a polynomial function.

We are now ready to state the convergence theorem for the exact penalty function using $\G1$.

\begin{thm}
Consider Problem \refe{nlp:eq} and define the following unconstrained problem
\begin{equation}
 \minimize_x \quad f(x) + \frac{\G1}{2}  \| h(x)\|^{2}.
\label{unconst:eq}
\end{equation}
Let
\[{x^*} = {x}^{*0} + \G1^{-1} {x}^{*1} + \G1^{-2} {x}^{*2}+\ldots\]
 be a stationary point for \refe{unconst:eq} and assume that the LICQ condition holds true at ${x}^{*0}$. Then,
the pair  $\left(x^{*0},{\pi^*}=  h^{(1)}({x^*}) \right)$ is a  KKT  point of \refe{nlp:eq}.
\end{thm}
\begin{proof}
Since $x^* = {x}^{*0} + \G1^{-1} {x}^{*1} + \G1^{-2} {x}^{*2}+\ldots$
is a stationary point we have that

\[ \nabla f(x^*) +  \G1  \sum_{j=1}^p \nabla h_j(x^*)h_j(x^*)= 0 \]

Therefore, from Assumption \ref{assumpt:expand}
\begin{eqnarray*}
 \nabla f(x^{*0})+ \G1^{-1} F^{(1)} (x^*) + \G1^{-2}F^{(2)}(x^*) + \ldots + \\
  + \ds \G1 \sum_{j=1}^p \Biggl [ \left ( \nabla h_j(x^{*0})+ \G1^{-1} H_j^{(1)} (x^*) + \G1^{-2}H_j^{(2)}(x^*) + \ldots \right) \\ \ds\hspace*{5em} \left (
h_j(x^{*0}) + \G1^{-1} h_j^{(1)} (x^*) + \G1^{-2}h_j^{(2)}(x^*) +\ldots
\right) \Biggr ] & = & 0
\end{eqnarray*}
and rearranging the terms we obtain:
\begin{eqnarray*}
\G1 \Biggl ( \sum_{j=1}^p  \nabla h_j(x^{*0})h_j(x^{*0}) \Biggr ) + \\
+ \Biggl (
       \nabla f(x^{*0})
   + \sum_{j=1}^p  \nabla h_j(x^{*0}) h_j^{(1)} (x^*) + \sum_{j=1}^p  H_j^{(1)} (x^*) h_j(x^{*0})
\Biggr )  + \\
+ \G1^{-1} \Biggl(\ldots \Biggr ) + \G1^{-2} \Biggl ( \ldots \Biggr ) + .... & = & 0
\end{eqnarray*}
Hence,
\[  \sum_{j=1}^p  \nabla h_j(x^{*0})h_j(x^{*0}) = 0  \]
and, from the LICQ condition, it follows that
\begin{equation} h(x^{*0}) = 0. \label{eq:hfeas} \end{equation}
Therefore, the point  $x^{*0}$ is feasible for \refe{nlp:eq}. Moreover, from
 \[ \nabla f(x^{*0})
   + \sum_{j=1}^p  \nabla h_j(x^{*0}) h_j^{(1)} (x^*) + \sum_{j=1}^p  H_j^{(1)} (x^*) h_j(x^{*0})= 0\]
and \refe{eq:hfeas} above, it follows that
\begin{equation} \nabla f(x^{*0})
   + \sum_{j=1}^k  \nabla h_j(x^{*0}) h_j^{(1)} (x^*) = 0 \end{equation}
that shows that  $\left(x^{*0},{\pi^*}=  h^{(1)}({x^*})§ \right)$ is a  KKT  point of \refe{nlp:eq}.
\end{proof}

\subsubsection{A simple example}
Consider the following simple 2--dimensional optimization problem with a single linear constraint\footnote{The example is taken from \citep{Liuzzi:note}}:
\eprobmin{x}{\ds \frac{1}{2}x_{1}^{2}+\frac{1}{6}x_{2}^{2}}
     { x_{1}+x_{2}= 1 } {nlp:ex1_a}
The pair $\left(\xbar, \pibar\right)$ with $\xbar = \left [\ba{c} {1}/{4} \\ {3}/{4}\ea \right ]$, $\ds \pibar = -\frac{1}{4}$ is a KKT point.
The corresponding unconstrained optimization problem is
\begin{equation} \minimize_x \frac{1}{2}x_{1}^{2}+\frac{1}{6}x_{2}^{2}- \frac{1}{2}\G1(1- x_{1}-x_{2})^{2} \label{ex1:b} \end{equation}
and
the first Order Optimality Conditions are
\[
\left\{ \begin{array}{l} x_{1}- \G1(1- x_{1}-x_{2})= 0  \\
\frac{1}{3}x_{2} - \G1(1- x_{1}-x_{2})= 0 \\
 \end{array}
 \right . \]
Therefore, the point
\[     x^*_1 = \frac{1\G1}{1 + 4 \G1},\quad \quad  x^*_2 = \frac{3\G1}{1 + 4 \G1}\]
is a stationary point of Problem \refe{ex1:b}.
Note that
\[ \ba{ccl}    x^*_1 & = &  \frac{1}{4}- \G1^{-1} \left( \frac{1}{16} - \frac{1}{64}\G1^{-1} \ldots \right ) \\ \\
x^*_2 & = &\frac{3}{4}- \G1^{-1} \left( \frac{3}{16} - \frac{3}{64}\G1^{-1}\ldots \right )  \ea \]
and $x^{*0}_1 = \frac{1}{4} = \xbar_1$ and $x^{*0}_2 = \frac{3}{4} = \xbar_2$. Moreover,
\begin{align*}   {-\G1(1- x^*_{1}-x^*_{2})}
  = & -\G1 \Bigl(1- \frac{1}{4}+ \G1^{-1} \frac{1}{16} - \frac{1}{64}\G1^{-2}\ldots  \\
  & \hspace*{5em} -\frac{3}{4}+ \G1^{-1} \frac{3}{16} - \frac{3}{64}\G1^{-2}\ldots \Bigr) \\
   = & -\G1\left( \G1^{-1} \frac{1}{16}+ \G1^{-1}\frac{3}{16} - \G1^{-2}\frac{4}{64}\ldots \right) \\
   = & -\frac{1}{4}+ \frac{4}{64} \G1 ^{-1}\ldots
  \end{align*}
and $ h^{(1)}({x^*}) = -\frac{1}{4} = \pibar$

\subsection{The general constrained case.} \label{subsect:ineq}
In this subsection we return back to the most general nonlinear optimization problem \refe{nlp:gen} with equality and inequality constraints and we introduce the following corresponding unconstrained optimization problem:
\begin{equation}
\min_x f(x) + \frac{1}{2}\G1  \left\| \max\{0, g_{i}(x)\}\right\|^{2} + \frac{1}{2}\G1  \left\| h(x) \right\|^{2}
\label{unconst:ineq} \end{equation}

In addition to the conditions stated in Assumption \ref{assumpt:expand} we similarly require that, under the same hypothesis on $x$, the following conditions hold true:

\begin{eqnarray*}
 g(x) &= & g(x^0) + \G1^{-1} g^{(1)} (x) + \G1^{-2}g^{(2)}(x) +\ldots \\
\nabla g(x) & = & \nabla g(x^{0})+ \G1^{-1} G^{(1)} (x) + \G1^{-2}G^{(2)}(x) + \ldots
\end{eqnarray*}

In order to derive optimality conditions we need a modified version of the LICQ condition.

\begin{definition}\label{def:LICQmod}
Let $x^0 \in \R^n$. The Modified LICQ (MLICQ) condition is said to hold true at $x^0$ if the  vectors $\Bigl \{\grad g_i(x^0)_{ i:g_i(x^0)\geq0},  \grad h_j(x^0)_{j = \oneto{k}} \Bigr\}$ are linearly independent.
\end{definition}

The following theorem shows the relationship between stationary points of Problem \refe{unconst:ineq} and KKT points of the general optimization Problem \refe{nlp:gen}.
\begin{thm}
Consider the Problem \refe{nlp:gen} and the corresponding unconstrained Problem \refe{unconst:ineq}.
Let
\[{x^*} = {x}^{*0} + \G1^{-1} {x}^{*1} + \G1^{-2} {x}^{*2}+\ldots\]
 be a stationary point for \refe{unconst:ineq} and assume that MLICQ condition holds true at ${x}^{*0}$. Then,
the triplet  $ \Bigl (x^{*0}, {\mu^*}, {\pi^*}= h^{(1)}({x^*}) \Bigr)$ is a  KKT  point of \refe{nlp:gen}{ where
\[ {\mu_i^*}= \left \{ \ba{ll} 0 & \insmat{ if } g_i(x^{*0}) < 0 \\
                          \max \left \{0, g_i^{(1)}( x^* ) \right \} & \insmat{ if } g_i(x^{*0}) = 0 \ea \right . \]}
\end{thm}
\begin{proof}
Since $x^* = {x}^{*0} + \G1^{-1} {x}^{*1} + \G1^{-2} {x}^{*2}+\ldots$ is a stationary point we have

\[ \nabla f(x^*) + \G1 \sum_{i=1}^m\nabla g_{i}(x^*)\max \left \{0, g_{i}(x^*) \right \} + \G1  \sum_{j=1}^p \nabla h_j(x^*)h_j(x^*)= 0. \]
From Assumption \ref{assumpt:expand} we obtain
\[
\ba{lcc}
 \ds \nabla f(x^{*0})+ \G1^{-1} F^{(1)} (x^*) + \G1^{-2}F^{(2)}(x^*) + \ldots + \\
+ \ds \G1 \sum_{i=1}^m  \Biggl [ \left ( \nabla g_i(x^{*0})+ \G1^{-1} G_i^{(1)} (x^*) + \G1^{-2}G_i^{(2)}(x^*) + \ldots \right) \\ \ds\hspace*{5em} \left (
\max \left \{0, g_i(x^{*0}) + \G1^{-1} g_i^{(1)} (x^*) + \G1^{-2}g_i^{(2)}(x^*) +\ldots \right \}
\right)\Biggr ] + \\
+ \ds \G1 \sum_{j=1}^p \Biggl [ \left ( \nabla h_j(x^{*0})+ \G1^{-1} H_j^{(1)} (x^*) + \G1^{-2}H_j^{(2)}(x^*) + \ldots \right) \\ \ds\hspace*{5em} \left (
h_j(x^{*0}) + \G1^{-1} h_j^{(1)} (x^*) + \G1^{-2}h_j^{(2)}(x^*) +\ldots
\right) \Biggr ]& = & 0
\ea \]
Note that the following implications hold:
\begin{eqnarray*}
g_i(x^{*0}) > 0 &  \Rightarrow & \max \left \{0, g_{i}(x^*) \right \} =
g_i(x^{*0}) + \G1^{-1} g_i^{(1)}( x^* )+ \G1^{-2} g_i^{(2)}(x^*) \ldots  \\
g_i(x^{*0})  < 0 &\Rightarrow & \max \left \{0, g_{i}(x^*) \right \} = 0 \\
g_i(x^{*0})  = 0 &\Rightarrow & \max \left \{0, g_{i}(x^*) \right \} = \G1^{-1}
\max \left \{0, g_i^{(1)}(x^*)+ \G1^{-1} g_i^{(2)}(x^*)+ \ldots \right \}.
\end{eqnarray*}
Therefore,
\[
\ba{lcc}
 \ds \nabla f(x^{*0})+ \G1^{-1} F^{(1)} (x^*) + \G1^{-2}F^{(2)}(x^*) + \ldots + \\
+ \ds \G1 \sum_{\substack{i=1\\{g_i(x^{*0}) > 0}}}^m  \Biggl [ \left ( \nabla g_i(x^{*0})+ \G1^{-1} G_i^{(1)} (x^*) + \G1^{-2}G_i^{(2)}(x^*) + \ldots \right) \\ \ds\hspace*{5em} \left (
                 g_i(x^{*0}) + \G1^{-1} g_i^{(1)} (x^*) + \G1^{-2}g_i^{(2)}(x^*) +\ldots
\right)\Biggr ] + \\
+ \ds \G1 \sum_{\substack{i=1\\{g_i(x^{*0}) = 0}}}^m  \Biggl [ \left ( \nabla g_i(x^{*0})+ \G1^{-1} G_i^{(1)} (x^*) + \G1^{-2}G_i^{(2)}(x^*) + \ldots \right) \\ \ds\hspace*{5em} \left (\G1^{-1} \max \left\{0,   g_i^{(1)}(x^*)+ \G1^{-1} g_i^{(2)}(x^*)+ \ldots \right \}
\right)\Biggr ] + \\
+ \ds \G1 \sum_{j=1}^p \Biggl [ \left ( \nabla h_j(x^{*0})+ \G1^{-1} H_j^{(1)} (x^*) + \G1^{-2}H_j^{(2)}(x^*) + \ldots \right) \\ \ds\hspace*{5em} \left (
h_j(x^{*0}) + \G1^{-1} h_j^{(1)} (x^*) + \G1^{-2}h_j^{(2)}(x^*) +\ldots
\right) \Biggr ]& = & 0
\ea \]
Similarly to the equality constraint case, after rearranging terms we obtain:
\[ \sum_{\substack{i=1\\{g_i(x^{*0}) \geq 0}}}^m  \nabla g_i(x^{*0})  g_i(x^{*0}) +
       \sum_{j=1}^p  \nabla h_j(x^{*0})h_j(x^{*0}) = 0 \]
and, since MLICQ holds at $x^{*0}$, $g(x^{*0}) \leq 0$ and $h(x^{*0}) = 0$.
Morever, taking into account that $g(x^{*0}) \leq 0$, we obtain that
\[ \nabla f(x^{*0}) +
       \sum_{\substack{i=1\\{g_i(x^{*0}) = 0}}}^m \nabla g_i(x^{*0})  \max \left \{0, g_i^{(1)}( x^* ) \right \}
   +   \sum_{j=1}^p  \nabla h_j(x^{*0}) h_j^{(1)} (x^*)  = 0 \]
that completes the proof.
\end{proof}
It is worth to note that, when $g_i(x^{*0})  = 0$, it is still possible that  $ g_i^{(1)}(x^*)+ \G1^{-1} g_i^{(2)}(x^*)+ \ldots > 0$ and, therefore, the quantity $\ds \sum_{\substack{i=1\\{g_i(x^{*0}) = 0}}}^m \nabla g_i(x^{*0}) g_i^{(1)}(x^*)$ is present in the last equation of the proof.

\subsubsection{A second example}
Also in this case we present a very simple example to clarify the previous result.
Consider the minimization problem
\eprobmin{x\in \R}{x}{x \geq 1}{ex2:a}
(\ie $g(x) = 1-x$)  for which the optimal solution is $\xbar = 1$ with associated multiplier $\mubar=1$. The corresponding unconstrained optimization problem is
\begin{equation} \minimize_{x \in \R} x + \frac{\G1}{2} \insmat{max}\left \{0, 1-x \right \}^2. \label{ex2:b} \end{equation}
The first order optimality condition is
\[ 1 - \G1 \insmat{max}\left \{0, 1-x \right \} = 0. \]
The above equation has no solution when $1-x \leq 0$ and for $x < 1$ the only solution is
\[ x^* = \frac{\G1 - 1}{\G1} = 1 - \G1^{-1}  \]
Therefore, as expected, $x^{*0} = 1$. Moreover, $g(x^*) = 1 -  \left ( 1 - \G1^{-1}\right ) = \G1^{-1}$ and $\mu^* = 1$.

\section{Conclusions} \label{sect_conclusions}
In this paper we have presented possible uses of \G1 in Linear and Nonlinear Programming. In the new numerical system, making full use of \G1, it is possible to implement,  in a very simple way, anti--cycling strategies for  the simplex method. Moreover, we have shown that exact, differentiable penalty methods can be constructed for general  nonlinear programming. These are not the only possible applications of \G1 in Operations Research and Mathematical Programming. Another  interesting application is in Data Envelopment Analysis (DEA) methodology.  In the basic version proposed by Charnes, Cooper and Rhodes \citep{ccr:78} the use of a infinitesimal non--archimedean quantity $\epsilon$ is proposed. Negative power of \G1 will allow to achieve the same theoretical results and, thus, the efficiency of a single Decision Making Unit (DMU)  can be easily obtained by solving a single Linear Programming problems using the new arithmetic based on \G1.
\section*{Acknowledges}
The authors would like to thank Prof. Gaetano Zanghirati and Giada Galeotti for their useful suggestions on the definition of the multipliers for the general nonlinear case.







\end{document}